\documentclass[a4paper, 11pt]{article}

\usepackage
{
makeidx,epsfig, amstext, setspace, amsmath, amssymb,amsthm,amsfonts,times,latexsym,mathptmx,algorithm,algorithmic,
wrapfig,bbm,graphicx,color,enumerate,endnotes
}
\usepackage{tabularx}

\usepackage[top=2.5cm, bottom= 2.5cm, left= 2.5cm, right= 2.5cm]{geometry}
\newtheorem{theorem}{Theorem}

\newtheorem{LE}[theorem]{Lemma}

\newtheorem{OB}{\bf Observation}

\newtheorem*{DEF}{Definition}

\newcounter{claim_nb}[theorem]
\newtheorem{claim}[claim_nb]{Claim}
\newtheorem*{claim*}{Claim}

\newcommand{\zP}{\mathcal P}
\newcommand{\zQ}{\mathcal Q}

\newcommand{\zX}{\mathcal X}
\newcommand{\zY}{\mathcal Y}

\newcommand{\ignore}[1]{}
 
\newenvironment{cproof}
{\begin{proof}
 [Proof.]
 \vspace{-1.5\parsep}
}
{ \end{proof}}

\begin{document}
\title{The structure of graphs not admitting a fixed immersion}
\author{
Paul Wollan 
\thanks{
Department of Computer Science, University of Rome, ``La Sapienza", Rome, Italy \texttt{wollan@di.uniroma1.it}. 
Supported by the European Research Council under the European Union's Seventh Framework Programme (FP7/2007-2013)/ERC Grant Agreement no. 279558.}
\\
}
\maketitle
%

\begin{abstract}
We present an easy structure theorem for graphs which do not admit an immersion of the complete graph $K_t$.  The theorem motivates the definition of a variation of tree decompositions based on edge cuts instead of vertex cuts which we call tree-cut decompositions.  We give a definition for the width of tree-cut decompositions, and using this definition along with the structure theorem for excluded clique immersions, we prove that every graph either has bounded tree-cut width or admits an immersion of a large wall.
\end{abstract}
%
\section{Introduction}

The graphs we consider in this article may have multiple edges but no loops.  In this article, we consider the immersion containment relation on graphs.  
\begin{DEF}
A graph $G$ admits a \emph{weak immersion} of a graph $H$ if there exist functions $\pi_v: V(H) \rightarrow V(G)$ and $\pi_e$ mapping the edges of $H$ to subgraphs of $G$ satisfying the following:
\begin{itemize}
\item[a.] the map $\pi_v$ is an injection;
\item[b.] for every edge $f \in E(H)$ with endpoints $x$ and $y$, $\pi_e(f)$ is a path with endpoints equal to $\pi_v(x)$ and $\pi_v(y)$;
\item[c.] for edges $f, f' \in E(H)$, $f \neq f'$, $\pi_e(f)$ and $\pi_e(f')$ have no edge in common.
\end{itemize}
We say that $G$ admits a \emph{strong immersion} of $H$ if the following condition holds as well.
\begin{itemize}
\item[d.] For every edge $f \in E(H)$ with endpoints $x$ and $y$, the path $\pi_e(f)$ intersects the set $\pi_v(V(H))$ only in its endpoints.  
\end{itemize}
The vertices $\{\pi_v(x): x \in V(H)\}$ are the \emph{branch vertices} of the immersion.  We will also say that $G$ immerses $H$ or alternatively that $G$ contains $H$ as an immersion.  The edge-disjoint paths $\pi_e(f)$ for $f\in E(H)$ are the \emph{composite paths} of the immersion.
\end{DEF}
We will focus almost exclusively in this article on weak immersions.  In the interest of brevity, we will often refer to weak immersions simply as immersions.  Whenever we do consider strong immersions as well, we will always explicitly specify so.

Containment as an immersion is closely related to containment as a subdivision.  Recall that to \emph{suppress} a vertex $v$ of degree one or two in a graph $G$, we contract an edge $e$ incident with $v$ and delete any resulting loops.  The graph $G$ \emph{contains $H$ as a subdivision} if $H$ can be obtained from a subgraph of $G$ by repeatedly suppressing vertices of degree two.  Equivalently, $G$ contains $H$ as a subdivision if $G$ admits an immersion $(\pi_v, \pi_e)$ of $H$ such that for every pair of edges $f, f' \in E(H)$, the paths $\pi_e(f)$ and $\pi_e(f')$ are internally vertex-disjoint.  

We can alternately define weak immersions as follows.  Let $e_1$ and $e_2$ be edges in a graph $G$ such that the endpoints of $e_1$ are $x, y$ and the ends of $e_2$ are $y, z$ with $x$ and $z$ distinct.  To \emph{split off} the edges $e_1$ and $e_2$, we delete the edges $e_1$ and $e_2$ from $G$ and add a new edge $e$ with endpoints $z$ and $x$.  Then $G$ contains $H$ as a weak immersion if and only if $H$ can be obtained from a subgraph of $G$ by repeatedly splitting off pairs of edges and suppressing vertices of degree two.

We prove two structural results in this article.  First, we present an easy structure theorem for graphs excluding the complete graph $K_t$ as an immersion for fixed values of $t$.  The proof is quite short and seems to have been independently discovered before.   A qualitative version of this theorem was shown by Seymour at the 2003 PIMS ``Workshop on Structural Graph Theory" in Vancouver, but was never published.  Recently, DeVos, McDonald, Mohar, and Scheide have proven the structure theorem \cite{DeVos} with essentially the same bounds which we obtain here.  

The structural result for excluding a clique immersion gives rise to a natural decomposition similar to tree decompositions based on edge cuts instead of vertex cuts.  We call these decompositions tree-cut decompositions and give a definition for the width of a tree-cut decomposition.  The main result of this article is to show an analog for the grid minor theorem for these tree-cut decompositions.  We show that if a graph has sufficiently large tree-cut width, then it admits an immersion of an \emph{$r$-wall}, a graph similar to the $r \times r$-grid.  

The study of graph immersions has recently seen a flurry of attention.  Robertson and Seymour showed \cite{RS23} that graphs are well-quasi-ordered under weak immersion containment, confirming a conjecture of Nash-Williams \cite{N}.  DeVos et al.~\cite{DDFMMS} have calculated the correct (up to a multiplicative constant) extremal function for the number of edges forcing a clique immersion in simple graphs.  Ferrara et al.~\cite{FGTW} have instead calculated tight minimal degree conditions which suffice to ensure that a graph contains a fixed graph $H$ as an immersion.  In an alternate line of inquiry, researchers have looked at the relationship between the chromatic number of the graph and the presence of clique immersions.  Abu-Khzam and Langston \cite{AL} have modified the infamous Hadwiger's conjecture on the relationship between the chromatic number and the largest clique minor in a graph by conjecturing that every graph of chromatic number $t$ must admit an immersion of $K_t$.  The conjecture has been verified for small values of $t$ by \cite{LM} and independently by DeVos et al.~\cite{DKMO}.  Kawarabayashi and Kobayashi have shown good approximation bounds on coloring problems by excluding an immersion of a fixed clique in \cite{KK}.  Finally, recent work has considered exact characterizations of graphs which do not admit as an immersion small fixed graphs such as $K_{3,3}$ or $K_5$ \cite{D, KT}.

We conclude the section fixing some notation.  Let $G$ be a graph and $v \in V(G)$.  The degree $\deg(v)$ is the number of edges incident with $v$ and $\Delta(G)$ is the maximum degree of a vertex in $G$.  Let $X \subseteq V(G)$.  The set of edges with exactly one endpoint in $X$ is denoted $\delta(X)$.  We will use $\delta(v)$ for $\delta(\{v\})$.  For a subset $X$ of vertices, we refer to the graph induced on $X$ by $G[X]$.   We use $G-X$ to refer to the graph induced on $V(G) - X$.  For a subset $F \subseteq E(G)$ of edges, we use $G-F$ to refer to the graph $(V(G), E(G) \setminus F)$.  For subgraphs $G_1$ and $G_2$ of $G$, the subgraph $G_1 \cup G_2$ has vertex set $V(G_1) \cup V(G_2)$ and edge set $E(G_1) \cup E(G_2)$.  We will use $G-x$ as shorthand notation for $G-\{x\}$ when $x$ is a single element of either $V(G)$ or $E(G)$.  Finally, we will often want to reduce $G$ to a smaller graph by identifying a subset of vertices to a single vertex.  Let $X \subseteq V(G)$, define $G'$ be the graph obtained by deleting every edge with both endpoints in $X$ and identifying the vertex set $X$ to a single vertex.  We will say that $G'$ is obtained from $G$ by \emph{consolidating} $X$.


\section{Weak immersions and connectivity}

Immersions are closely related to general edge connectivity.  Consider the following example.  Define the graph $S_{l, n}$ to be the graph with $n+1$ vertices $x_1, \dots, x_n, y$ and $l$ parallel edges from $x_i$ to $y$ for all $1 \le i \le n$.  
\begin{figure}[htb]\label{fig1}
\begin{center}
\includegraphics[scale = .5]{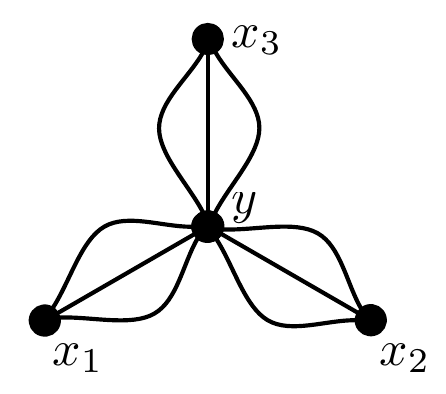}
 \caption{The graph $S_{3,3}$.}
 \end{center}
\end{figure}
See Figure \ref{fig1}.  The graphs $S_{k, n}$ have the property that they contain every fixed graph $H$ as an immersion for sufficiently large $k$ and $n$.  We formalize this in the next claim.

\begin{OB}\label{ob:multistar}
Let $H$ be a graph of maximum degree at most $k$ on $n$ vertices for positive integers $k$ and $n$.  Then the graph $S_{k, n}$ admits $H$ as an immersion. 
\end{OB}

The observation can be seen as follows.  Given the graph $H$, subdivide each edge of $H$, and identify all the new vertices of degree two to a single vertex $x$.  In the resulting graph, each vertex $v$ in $V(H)$ has exactly $\deg_H(v)$ parallel edges connecting it to $x$.  Thus, the resulting graph is a subgraph of $S_{k, n}$ for $n = |V(H)|$ and $k$ equal to the maximum degree of $H$.  Reversing this process shows how to arrive at $H$ by repeatedly splitting off edges. 

A consequence of Observation \ref{ob:multistar} is that for an arbitrary $H$, if a graph $G$ has sufficiently many vertices which are pairwise sufficiently edge connected, then $G$ admits an immersion of $H$.
\begin{LE}\label{lem:menger}
Let $t \ge 1$ be a positive integer, and $G$ a graph.  Assume there exists a subset $X \subseteq V(G)$, $|X| = t+1$, such that for every pair of vertices in $x, y \in X$, there does not exist an edge cut of order less than $t^2$ separating $x$ from $y$.  Then $G$ admits an immersion of $K_t$.
\end{LE}
\begin{proof}
By Observation \ref{ob:multistar}, we see that it suffices to find an immersion of $S_{t, t}$.  Label the vertices of $X$ as $x_1, \dots, x_t, y$.  We construct an auxiliary graph $G'$ obtained by adding a vertex $v$ with $t$ parallel edges connecting $v$ to $x_i$ for all $1 \le i \le t$.  If there exist $t^2$ edge-disjoint paths in $G'$ from $v$ to $y$, then by deleting the vertices $v$ from each path, we see that $G$ admits an immersion of $S_{t,t}$.  However, if there do not exist such paths in $G'$, then there exists a partition of $(X, Y)$ of $V(G')$ such that $v \in X$, $y \in Y$, and $\delta_{G'}(X) < t^2$.  Since $v$ has degree $t^2$, the set $X \setminus \{v\}$ contains at least one vertex $x_i$, and consequently, $(X\setminus\{v\}, Y)$ gives an edge cut in $G$ of order less than $t^2$ separating $x_i$ and $y$, contrary to our assumptions.
\end{proof}


\section{Edge sums}

Consider a graph $G$ which does not admit $K_t$ as an immersion for some fixed value $t$.  Lemma \ref{lem:menger} implies that if there are a large number of vertices in $G$ of degree at least $t^2$, then some pair of them must be separated by a bounded size edge cut.  This motivates the definition of a way to decompose a graph on edge cuts much in the same way that clique sums allow one to decompose a graph on vertex cuts.  We will refer to this operation as an edge sum.

\begin{DEF}
Let $G$, $G_1$, and $G_2$ be graphs.  Let $k \ge 1$ be a positive integer.  The graph $G$ is a \emph{$k$-edge sum} of $G_1$ and $G_2$ if the following holds.  There exist vertices $v_i \in V(G_i)$ such that $\deg(v_i) = k$ for $i = 1, 2$ and a bijection $\pi: \delta(v_1) \rightarrow \delta(v_2)$ such that $G$ is obtained from $(G_1 - v_1) \cup (G_2 - v_2)$ by adding an edge from $x \in V(G_1) - v_1$ to $y\in V(G_2) - v_2$ for every pair $e_1, e_2$ of edges satisfying $e_i \in \delta(v_i)$ for $i = 1, 2$, the ends of $e_1$ are $x$ and $v_1$, the ends of $e_2$ are $y$ and $v_2$, and $e_2 = \pi(e_1)$.

We will also refer to a $k$-edge sum as an edge sum of \emph{order} $k$.  The edge sum is \emph{grounded} if  there exist vertices $v_1'$ and $v_2'$ in $G_1$ and $G_2$, respectively, such that for $i = 1, 2$, $v_i' \neq v_i$ and there exist $k$ edge-disjoint paths linking $v_i$ and $v_i'$.  If $G$ can be obtained by a $k$-edge sum of $G_1$ and $G_2$, we write $G = G_1 \hat{\oplus}_k G_2$.
\end{DEF}

We first see that the operation of taking edge sums preserves the property of immersing a clique when the clique is larger than the order of the edge sum.

\begin{LE}\label{lem:sum2}
Let $G$, $G_1$, and $G_2$ be graphs and let $k, t \ge 1$ be positive integers with $t > k$.  Assume $G = G_1 \hat{\oplus}_k G_2$.  If $G$ admits an immersion of $K_t$, then either $G_1$ or $G_2$ does as well.  
\end{LE}

\begin{proof}
Let $Z$ be the set of branch vertices of an immersion of $K_t$ in $G$.  Let $X_i = V(G_i) \cap V(G)$ for $i = 1, 2$.  Observe that $|Z \cap X_i| \le 1$ for one of $i = 1, 2$ by the fact that that $\delta_G(X_i) = k < t$.  Thus, we may assume that all but one vertex of $Z$ is contained in $X_1$.  It follows that $G_1$ admits an immersion of $K_t$.  To see this, restrict the composite paths of the immersion to the edge set of $G_1$ and let the vertex of $V(G_1) \setminus X_1$ be a branch vertex in the case when $X_2$ contains a single vertex of $Z$.   Note that it is possible that the original immersion is strong, but the immersion we find in $G_1$ is weak, specifically, if $Z$ contains a vertex of $X_2$ and several of the composite paths of the immersion intersect $G[X_2]$ as well.  \end{proof}

We now see when the converse holds.  If a graph is an edge sum of two smaller graphs, and if the edge sum is grounded, then immersions in one of the smaller graphs extend readily to immersions in the larger graph.  We omit the proof.  

\begin{LE}\label{lem:sum}
Let $G$, $G_1$, and $G_2$ be graphs and let $k \ge 1$ be a positive integer.  Assume $G = G_1 \hat{\oplus}_k G_2$, and assume that the edge sum is grounded.  Let $H$ be an arbitrary graph.  If $G_1$ or $G_2$ admits an immersion of $H$, then $G$ does as well.  If the immersion in either $G_1$ or $G_2$ is strong, then the immersion in $G$ is also strong.  
\end{LE}

We now  combine the definition of edge sums along with Lemma \ref{lem:menger} to get a decomposition for graphs which do not admit a fixed clique as an immersion.  

\begin{DEF}
Let $G$ be a graph and $\alpha, \beta \ge 0$ positive integers.  Then $G$ has \emph{($\alpha, \beta)$-bounded degree} if there exist at most $\alpha$ vertices of degree at least $\beta$.
\end{DEF}

\begin{theorem}\label{thm:structweak}
Let $t \ge 1$ be a positive integer.  If $G$ is a graph which does not admit $K_t$ as a weak immersion, then either $G$ has $(t,t^2)$-bounded degree or there exist graphs $G_1$, $G_2$ which do not have an immersion of $K_t$, and an integer $k <  t^2$ such that $G$ is given by a grounded edge sum $G_1 \hat{\oplus}_k G_2$.  Moreover, $|V(G_1)|, |V(G_2)| < |V(G)|$.  
\end{theorem}

\begin{proof}
Let $G$ and $t$ be given.  Let $Z$ be the set of vertices of degree at least $t^2$.  If there exist two vertices $u$ and $v$ in $Z$ such that there do not exist $t^2$ edge-disjoint $u-v$ paths, then there exists a set $X \subseteq V(G)$ such that $u \in X$ and $v \in V(G) \setminus X$ such that $|\delta(X)| \le t^2-1$.  Choose such an $X$ to minimize $|\delta(X)|$.  Let $G_1$ be the graph obtained by consolidating $V(G) \setminus X$ and similarly, let $G_2$ be obtained from $G$ by consolidating $X$.   We see that $G = G_1 \hat{\oplus} G_2$ and that the order of the edge sum is at most $t^2 - 1$.  By our choice to minimize $|\delta(X)|$, there exist $|\delta(X)|$ edge-disjoint $u-v$ paths.  We conclude that the the edge sum is grounded, as required.  We see that $G_i$ does not contain an immersion of $K_t$ for $i = 1,2$ by Lemma \ref{lem:sum}.  Finally, $X \neq \{u\}$ and $V(G) \setminus X \neq \{v\}$ by the degree of $u$ and $v$, and so $|V(G_1)|<|V(G)|$, $|V(G_2)| < |V(G)|$.

Thus, we may assume that every pair of vertices in $Z$ are linked by at least $t^2$ edge-disjoint paths.  If $|Z| \ge t+1$, then by Lemma \ref{lem:menger}, $G$ admits $K_t$ as an immersion, a contradiction.  Thus, $|Z| \le t$, completing the proof.
\end{proof}




\section{The decomposition defined by edge sums and a structure theorem}

In the same way that clique sums give rise in a natural way to tree decompositions, we see that edge sums likewise give rise to a natural decomposition.  A \emph{near-partition} of a set $X$ is a family of subsets $X_1, \dots, X_k$, possibly empty, such that $\bigcup_1^k X_i = X$ and $X_i \cap X_j = \emptyset$ for all $1 \le i < j \le k$.  

\begin{DEF}
A \emph{tree-cut decomposition} of a graph $G$ is a pair $(T, \zX)$ such that $T$ is a tree and $\zX = \{X_t \subseteq V(G): t \in V(T)\}$ is a near-partition of the vertices of $G$ indexed by the vertices of $T$.  For each edge $e = uv$ in $T$, $T-uv$ has exactly two components, namely $T_v$ and $T_u$ containing $v$ and $u$ respectively.  The \emph{adhesion} of the decomposition is 
\begin{equation*}
max_{uv \in E(T)} \left | \delta\left ( \bigcup_{t \in V(T_v)} X_t \right)\right |
\end{equation*}
when $T$ has at least one edge, and 0 otherwise.  
The sets $\{X_t: t \in V(T)\}$ are called the \emph{bags} of the decomposition.
\end{DEF}
Note that the definition allows bags to be empty.  

Certain special cases of tree-cut decompositions have already been considered, namely the cut-width of a graph and the carving-width.  The cut-width of a graph is defined as the minimum adhesion of a tree-cut decomposition $(T, \zX)$ where $T$ is a path and every set of $\zX$ has size at most one.  Cut-width was originally studied as {\sc MINIMUM CUT LINEAR ARRANGEMENT} \cite{GJ}.  Much of the work has approached the problem from an algorithmic perspective, calculating the cut-width in specific classes of graphs as well as developing parameterized algorithms for calculating the cut-width in general.  See \cite{CS, FL, HLMP, TSB1, TSB2, Y}.  

Carving-width was introduced by Seymour and Thomas \cite{ST} and is in a certain sense analogous to the branch width of a graph.  The carving-width is the minimum adhesion of a tree-cut decomposition $(T, \zX)$ satisfying the properties that $T$ has maximum degree $3$ and the only non-empty bags of the decomposition are the leaves of $T$. 

Let $G$ be a graph and $(T, \zX)$ a tree-cut decomposition of $G$.  Fix a vertex $t \in V(T)$.  The \emph{torso of $(G, T, \zX)$ at $t$} is the graph $H$ defined as follows.  If $|V(T)| = 1$, then the torso $H$ of $(G, T, \zX)$ at $t$ is simply $G$ itself.  If $|V(T)| \ge 2$, let the components of $T-t$ be $T_1, \dots, T_l$ for some positive integer $l$.  Let $Z_i = \bigcup_{x \in V(T_i)} X_x$ for $1 \le i \le l$.  Then $H$ is made by consolidating each set $Z_i$ to a single vertex $z_i$.  The vertices $X_t$ are called the \emph{core vertices} of the torso.  The vertices $z_i$ are called the \emph{peripheral vertices} of the torso.  When there can be no confusion as to the graph $G$ in question, we will also refer to the torso of $(T, \zX)$ at a vertex $t$.

\begin{LE}\label{lem:tredecom}
Let $G$, $G_1$, and $G_2$ be graphs such that $G = G_1 \hat{\oplus}_k G_2$ for some $k \ge 0$.  If $G_i$ has a tree-cut decomposition $(T_i, \zX_i)$ for $i = 1, 2$, then $G$ has a tree-cut decomposition $(T, \zY)$ such that the adhesion of $(T, \zY)$ is equal to 
\begin{equation*}
max \{ k, adhesion(T_1, \zX_1), adhesion(T_2, \zX_2)\}.
\end{equation*}
Moreover, for every $t \in V(T)$, there exists $i \in \{1,2\}$ and a vertex $t'$ in $V(T_i)$ such that the torso $H_t$ of $(G, T, \zY)$ at $t$ is isomorphic to the torso $H'$ of $(G_i, T_i, \zX_i)$ at $t'$.  Finally, every core vertex of $H_t$ is a core vertex of $H'$.
\end{LE}
\begin{proof}
Let $v_i \in V(G_i)$ be the vertex of degree $k$ for $i \in \{1,2\}$ such that $G$ is obtained by identifying the edges of $\delta(v_1)$ and $\delta(v_2)$.  For each $i = 1, 2$, there exists a vertex $t_i$ of $T_i$ such that $v_i \in X_{t_i}$.  We construct a tree-cut decomposition of $G$ as follows.  The tree $T$ is defined to be the disjoint union of the trees $T_1$ and $T_2$ along with an additional edge from $t_1$ to $t_2$.  For every vertex $t \in V(T)$, $t$ is a vertex of either $T_1$ or $T_2$.  Let $Y_t$ be defined as the corresponding set $X_t \setminus \{v_1, v_2\}$.  Then $\zY = \{Y_t: t \in V(T)\}$ is a near-partition of the vertices of $G$.  We claim $(T, \zY)$ is the desired decomposition.

For an edge $tt'$ of $T$, let $T_t$ and $T_{t'}$ be the two components of $T-tt'$ containing $t$ and $t'$, respectively.  Unless $tt'$ is equal to the edge $t_1t_2$, we may assume without loss of generality that $T_t$ is a proper subtree of $T_1$.  Thus, for all $x \in V(T_t)$, $Y_x = X_x$.  It follows that the cut $\bigcup_{x \in V(T_t)} Y_x$ has the same order as the cut $\bigcup_{x \in V(T_t)} X_x$, as desired.  If $tt'$ is equal to the edge $t_1t_2$, the corresponding cut in $G$ has order $k$.  We conclude that the decomposition $(T, \zY)$ has the desired adhesion.

By the construction of $(T, \zY)$, we see that the torso of $(G, T, \zY)$ at any vertex of $T$ is equal to the torso of the corresponding vertex of $T_1$ or $T_2$ of $(G_i, T_i, \zX_i)$.  Moreover, the set of core vertices is the same except in the case of the torsos at the two vertices $t_1$ and $t_2$.  There, the vertices $v_1$ and $v_2$ are no longer core vertices but every other core vertex of the torso in $(G_i, T_i, \zX_i)$ remains a core vertex of $(G, T, \zY)$.  This completes the proof of the claim.
\end{proof}

We can now state the structure theorem for graphs excluding a fixed clique immersion in terms of a tree-cut decomposition.

\begin{theorem}\label{thm:weakdecomp2}
Let $G$ be a graph and $t\ge 1$ a positive integer.  If $G$ does not admit $K_t$ as a weak immersion, then there exists a tree-cut decomposition $(T, \zX)$ of $G$ of adhesion less than $t^2$ such that each torso has $(t,t^2)$-bounded degree.
\end{theorem}

\begin{proof}
We proceed by induction on $|V(G)|$.  Assume as a case that there exist $G_1$, $G_2$, and an integer $k$ such that $G = G_1 \hat{\oplus}_k G_2$.  By Lemma \ref{lem:sum}, we may assume that neither $G_1$ nor $G_2$ admits an immersion of $K_t$.  By induction, each of $G_1$ and $G_2$ has a tree-cut decomposition $(T_i, \zX_i)$ of adhesion less than $t^2$ such that each torso of $(G_i, T_i, \zX_i)$ has $(t,t^2)$-bounded degree.  By Lemma \ref{lem:tredecom}, we see that $G$ as well has a tree-cut decomposition such that every torso has $(t,t^2)$-bounded degree, as desired. Note that here we are using the fact that the torsos of the decomposition of $G$ are isomorphic to the respective torsos of the decompositions of each of $G_1$ and $G_2$.  

Thus, by Theorem \ref{thm:structweak}, we may assume that $G$ has $(t,t^2)$-bounded degree.  The trivial tree-cut decomposition with all of $G$ in a single bag satisfies the statement of the theorem, completing the proof.
\end{proof}

The previous theorem states that if a graph does not admit $K_t$ as an immersion then it has a certain decomposition. The converse statement is clearly not true.  The graph $K_t$ itself trivially admits a tree-cut decomposition where each torso has $(t,t^2)$-bounded degree, namely by including every vertex of the $K_t$ in a single bag (the torso actually has $(0,t)$-bounded degree).  In the next theorem, we see however that the converse is approximately true in that any graph admitting a tree-cut decomposition of adhesion less than $r$ such that each torso has $(r,r)$-bounded degree does not admit $K_{r + 1}$ as an immersion.  

\begin{theorem}\label{thm:upper}
Let $G$ be a graph and $r \ge 1$ a positive integer.  If $G$ admits a tree-cut decomposition $(T, \zX)$ of adhesion less than $r$ such that every torso has $(r,r)$-bounded degree, then $G$ does not admit an immersion of $K_{r+1}$.
\end{theorem}

\begin{proof}
Assume, to reach a contradiction, that $G$ contains an immersion of $K_{r+1}$.  Note that between each pair of branch vertices of the immersion there exist $r$ edge-disjoint paths.  By the adhesion bound on our decomposition, we see that all the branch vertices must be contained in a single bag of the decomposition.  But then the torso of that bag must contain at least $r+1$ vertices of degree $r$, a contradiction.  
\end{proof}

Given that $(t, t^2)$-bounded degree implies $(t^2, t^2)$-bounded degree, Theorem \ref{thm:upper} implies that if a graph $G$ admits the structure given in Theorem \ref{thm:weakdecomp2}, then $G$ does not admit an immersion of $K_{t^2 +1}$.  


\section{The width of a tree-cut decomposition}

Given the definition of tree-cut decompositions, it is natural to ask when does a graph have a bounded width tree-cut decomposition.  However, this will require a suitable definition of the width of a tree-cut decomposition.  If we follow the model of tree decompositions, the most natural measure would be to require the bags to have bounded size.  However, this runs into an immediate problem.  Let $P_t$ be the graph obtained by adding $t-1$ parallel edges to each edge of a path on $t$ vertices.  Then $P_{t^2}$ contains contains $K_t$ as an immersion and should therefore have large tree-cut width.  At the same time, $P_t$ has a tree-cut decomposition where each bag has one vertex and every torso has at most 3 vertices.  We conclude that any suitable width measure for tree-cut decompositions must take into account the adhesion of the decomposition.  

Observe that considering the adhesion alone does not yield a satisfying definition of width.  Every graph $G$ of degree at most $k$ trivially admits a tree-cut decomposition $(T, \zX)$ of adhesion at most $k$ with the additional property that each bag contains at most one vertex. Let $T$ be the star with $|V(G)|$ leaves and assign one vertex of $G$ to each leaf.  Such a decomposition fails to distinguish between the large variety of graphs of bounded degree; for example such graphs include expander graphs which have tree width roughly linear in the total number of vertices.

Looking closer at the previous example, one notices that for the decomposition given, the center vertex of the star will have an arbitrarily large torso.  Thus, one might consider requiring both the adhesion and the torsos to have bounded size.  However, by the definition of torso this will require the tree $T$ of the decomposition $(T, \zX)$ to have degree at most $k$.  Thus, if we consider a tree-cut decomposition of the star on $n+1$ vertices and we impose a bound on the size of the torsos, we will force the adhesion to be arbitrarily large by choosing $n$ sufficiently large.  However, such large torsos arise due to a large number of pendant vertices in the torso.

Keeping these examples in mind, we will formulate a definition of the width of a tree-cut decomposition based on the adhesion and size of the torsos after accounting for vertices of degree one and two.  The width measure is based on what we call the 3-center of the torsos of the decomposition.  We present two equivalent ways of describing the 3-center.  We first give the definition based on immersions which, although more technical, will make several subsequent statements easier to prove.  The second way of describing the 3-center is based on repeatedly suppressing small degree vertices.  We present it as Lemma \ref{lem:altcenter}.

Given a graph $G$ and a subset $X \subseteq V(G)$, we define the \emph{3-center of $(G, X)$} as the maximum (with respect to containment as an immersion) $H$ such that there exists an immersion $H$ in $G$ given by maps $(\pi_v, \pi_e)$ with the property that 
\begin{equation}
X \subseteq \pi_v(V(H)) \text{, and for every vertex }x \in V(H)\text{ with }deg_H(x) \le 2\text{ we have that }\pi_v(x) \in X.  
\end{equation}
We will also refer to the $3$-center of $(G, X)$ as the $3$-center of $G$ on $X$.

It is not immediately clear that the 3-center is well defined.  We now show that such a maximal $H$ as in the definition is unique.
\begin{LE}\label{lem:centerdef}
Let $G$ be a graph and $X \subseteq V(G)$.  There is a unique graph $H$ which is maximal with respect to containment as an immersion such that there exists an immersion of $H$ in $G$ $(\pi_v, \pi_e)$ satisfying (1).
\end{LE}

\begin{proof}
Assume there exist distinct $H_1$ and $H_2$ maximal with respect to containment as an immersion and a pair $(G, X)$ such that there exist maps $(\pi_v^i, \pi_e^i)$ defining the immersion for $i = 1, 2$ satisfying (1) for $H_1$ and $H_2$, respectively.  Moreover, assume that we choose such $H_1, H_2, G, (\pi_v^1, \pi_e^1), (\pi_v^2, \pi_e^2)$ to minimize $|V(G)|$.  

First, we see that for each $i = 1, 2$ and for every pair of edges $f, f' \in E(H_i)$, the paths $P = \pi_e^i(f)$ and $P' = \pi_e^i(f')$ do not have a common internal vertex.  Let $x \in V(G)$ be a common vertex in $P$ and $P'$ which is not an endpoint of either $P$ or $P'$.  Note that $x$ has degree at least 4 in $G$.  Let $H_i'$ be obtained by subdividing the edges $f$ and $f'$ and identifying the two new vertices to a vertex $z$.  Then $H_i'$ is contained as an immersion in $G$ by mapping $z$ to the vertex $x \in V(G)$.  Moreover, the immersion will satisfy (1) as the vertex $z$ which is mapped to a vertex of $V(G) \setminus X$ has degree four.  A similar argument shows that for every edge $f \in E(H_i)$ and for every vertex $x \in V(H_i)$, the path $\pi_e^i(f)$ does not contain $\pi_v^i(x)$ as an internal vertex.  

By our choice to minimize $|V(G)|$, there cannot exist a vertex $v$ of $V(G) \setminus X$ of degree at most two in $G$. To see this, by (1), the vertex $v$ cannot be a branch vertex of either immersion and so if $G'$ is obtained by contracting an edge incident with $v$, then both $H_1$ and $H_2$ are contained as an immersion in $G'$, contradicting our choice to minimize $|V(G)|$.  It follows that $|V(H_1)| = |V(H_2)| = |V(G)|$, and therefore the previous paragraph, we see that every edge of $H_i$ is mapped to an edge of $G$.  By the maximality of both $H_1$ and $H_2$, we have that $ H_1 = G = H_2$, a contradiction.  
\end{proof}

The 3-center of $(G, X)$ can be equivalently thought of as the unique graph obtained by repeatedly suppressing any vertex of $V(G) \setminus X$ of degree at most two.    
\begin{LE}\label{lem:altcenter}
Let $G$ be a graph and $X \subseteq V(G)$.  There is a unique graph $H$ such that $H$ is obtained by a maximal sequence of suppressing a vertex not in $X$ of degree at most two and deleting any resulting loops.  Moreover, $H$ is equal to the 3-center of $(G, X)$.
\end{LE}
\begin{proof}
It suffices to show that $H$ is equal to the 3-center of $(G, X)$ as uniqueness then follows from Lemma \ref{lem:centerdef}.  Assume the claim is false, and pick $(G, X)$ a counter-example minimizing $|V(G)|$.  Clearly, $G$ has at least one vertex of degree at most two in $V(G) \setminus X$.  Fix a maximal sequence of vertex suppressions of vertices of degree at most $2$ in $V(G) \setminus X$ which results in a graph $H$.  Let $v$ be the first vertex suppressed in the sequence, and let $G'$ be the graph obtained by suppressing $v$ and deleting any resulting loops.  By the minimality of $G$, $H$ is equal to the $3$-center of $(G', X)$.  Moreover, $H$ immerses in $G$ while maintaining the property that no branch vertex in $V(G) \setminus X$ has degree at most two.  If $H$ is not also the $3$-center of $(G, X)$, there exists an immersion of a graph $\bar{H}$ which strictly contains $H$ as an immersion with no branch vertices of degree less than $3$ in $V(G)\setminus X$.  The immersion in $G$ yields an immersion in $G'$ as well, yielding a contradiction to the fact that $H$ is the 3-center of $G$.  This completes the proof.
\end{proof}

Note that by the definition of the 3-center of a pair $(G, X)$, for any set $Y \subseteq X$, the $3$-center of $(G, Y)$ is contained as an immersion in the $3$-center of the pair $(G, X)$. 

We now give the definition of the width of a tree-cut decomposition.
\begin{DEF}
Let $G$ be a graph and $(T, \zX)$ a tree-cut decomposition of $G$.  For each vertex $t \in V(T)$, let $X_t$ be the bag at the vertex $t$.   Let $H_t$ be the torso of $(T, \zX)$ at $t$, and let $\bar{H}_t$ be the 3-center of $(H_t, X_t)$.  Let $\alpha$ be the adhesion of the decomposition.  The \emph{width} of the decomposition is 
\begin{equation*}
max \left [ \{\alpha\} \cup \{|V(\bar{H}_t)|: t \in V(T)\} \right ].
\end{equation*}
The \emph{tree-cut width} of the graph $G$, also written $tcw(G)$, is the minimum width of a tree-cut decomposition.
\end{DEF}

It is easy to see that any tree has tree-cut width one, and that a cycle has tree-cut width two. Note as well that if a graph is 3-edge connected, then the 3-center of $(H_t, X_t)$ is simply the graph $H_t$ for every torso $H_t$ at a vertex $t$ in the decomposition.  In effect, the usage of the 3-centers is simply to ensure that vertices of degree one and two don't have the effect of blowing up the 3-center at any vertex of $T$.

We now prove two basic properties of the tree-cut width, namely that it is preserved under taking immersions and under edge sums.

\begin{LE}\label{lem:bdedimwidth}
Let $G$ and $H$ be graphs such that $G$ admits an immersion of $H$.  Then $tcw(H) \le tcw(G)$.
\end{LE}

\begin{proof}
Clearly the tree-cut width cannot increase upon deleting an edge or an isolated vertex.  Thus, it suffices to show that the statement holds when there exist edges $xy$ and $yz$ in $G$ and $G' = (G- \{xy, yz\} )+ xz$ for distinct vertices $x$, $y$, and $z$.  Let $(T, \zX)$ be a tree-cut decomposition of $G$ of minimum width.  As $V(G') = V(G)$, $(T, \zX)$ is a tree-cut decomposition of $G$ as well.  For any edge $e \in E(T)$, the number of edges in $G'$ crossing the cut defined by $e$ is at most the number of edges in $G$ crossing the cut defined by $e$.  Thus, the adhesion of the decomposition $(T, \zX)$ as a decomposition of $G'$ is at most the adhesion when considered a decomposition of $G$.  

Fix a vertex $t \in V(T)$, and let $H$ be the torso of $(G, T, \zX)$ at $t$ and let $H'$ be the torso of $(G', T, \zX)$ at $t$.  Note that $V(H) = V(H')$.  Every vertex of $H$ corresponds to a non-empty subset of the vertices of $G$ (possibly just a single vertex).  If we consider the possible cases for how the vertices $\{x, y, z\}$ can be split among these subsets, we see that either $E(H)  \subseteq E(H')$, or alternatively, $H'$ is obtained from $H$ by splitting off two incident edges.  In either case, we see that any immersion of a graph $J$ in $H'$ satisfying (1) will be an immersion in $H$ satisfying (1) as well.   Thus, the 3-center of $(H', X_t)$ has at most as many vertices as the 3-center of $(H, X_t)$.  We conclude that the tree-cut width of $G'$ is at most the tree-cut width of $G$, as desired.
\end{proof}

\begin{LE}\label{lem:bdwidthsum}
Let $G$, $G_1$, and $G_2$ be graphs and let $k \ge 1$ be a positive integer.  Assume $G = G_1 \hat{\oplus}_k G_2$.  If $G_1$ and $G_2$ each have tree-cut width at most $w$ for $w \ge k$, then $G$ has tree-cut width at most $w$.
\end{LE}

\begin{proof}
Let $(T_i, \zX_i)$ be a tree-cut decomposition of $G_i$ of width at most $w$ for $i = 1, 2$.  Lemma \ref{lem:tredecom} implies that there exists a tree-cut decomposition $(T, \zX)$ of $G$ of adhesion at most $w$ such that for every $t \in V(T)$, there exists a vertex $t'$ in $V(T_i)$ for one of $i = 1, 2$ satisfying the following.  Let $X_t \in \zX$ be the bag corresponding to $t$ and $H$ the torso of $(G, T, \zX)$ at $t$, and let $X_{t'} \in \zX_i$ be the bag of $(T_i, \zX_i)$ corresponding to $t'$ and $H'$ the torso of $(G_i, T_i, \zX_i)$ at $t'$.  Then $H = H'$ and $X_{t} \subseteq X_{t'}$.   It follows that the 3-center of $(H, X_{t})$ is contained as an immersion in the 3-center of $(H', X_{t'})$.   We conclude that the width of $(T, \zX)$ is at most $w$, as desired.  
\end{proof}

We will now show that if a graph has both bounded degree and bounded tree width, then it has bounded tree-cut width.  We first need the definition of a tree decomposition.

\begin{DEF}
A \emph{tree decomposition} of a graph $G$ is a pair $(T, \zX)$ such that $T$ is a tree and $\zX = \{X_t \subseteq V(G): t \in V(T)\}$ are subsets of $V(G)$ indexed by the vertices of $T$.  Moreover, we require that the subsets $\zX$ satisfy the following:
\begin{enumerate}
\item $\bigcup_{\{t \in V(T)\}} X_t = V(G)$ and for every edge $e = uv$ of $G$, there exists $t\in V(T)$ such that $\{u, v\} \subseteq X_t$.
\item for every vertex $v \in V(G)$, the set $\{t \in V(T): v \in X_t\}$ induces a connected subtree of $T$.
\end{enumerate}
The \emph{width} of the decomposition is $max_{\{t \in V(T)\}} |X_t|-1$ and the \emph{tree width} of a graph is the minimum width of a tree decomposition.  
\end{DEF}

\begin{LE}\label{lem:bdedtw}
Let $w, d \ge 1$ be positive integers and let $G$ be a graph with $\Delta(G) \le d$ and  tree width at most $w$. Then there exists a tree-cut decomposition of adhesion at most $(2w+2)d$ such that every torso has at most $(d+1)(w+1)$ vertices.  Specifically, the tree-cut width of $G$ is at most $(2w+2)d$.
\end{LE}

\begin{proof}
We may assume $G$ is connected.  Let $(T, \zX)$ be a tree decomposition of $G$ of width at most $w$.  We may assume that for any edge $tt'$ of $T$, $X_t \nsubseteq X_{t'}$ and $X_{t'} \nsubseteq X_{t}$.  Thus, for any vertex $t \in V(T)$, if we let $T_1, \dots, T_k$ be the components of $T-t$ then we may choose vertices $x_i \in \bigcup_{t' \in V(T_i)} X_{t'}$ such that $x_i \neq x_j$ for all $i \neq j$ and $x_i$ has a neighbor in $X_t$. Since $X_t$ has at most $w+1$ vertices and each vertex in $X_t$ has degree at most $d$, we see that $k \le d(w+1)$.  Thus, the tree $T$ has maximum degree $d(w+1)$.  

For every vertex $v \in V(G)$, we arbitrarily fix a vertex $t(v) \in V(T)$ such that $v \in X_{t(v)}$.  Let $X'_t = \{v \in V(G): t(v) = t\}$ for all vertices $t \in V(T)$ and let $\zX' = \{X_t': t \in V(T)\}$.  By construction, $X'_t \subseteq X_t$ for all $t \in V(T)$.  We will see that the tree-cut decomposition $(T, \zX')$ has width at most $(2w+2)d$.  We first observe that since $T$ has degree at most $d(w+1)$ and $|X_t'| \le w+1$, then the size of each torso is at most $(d+1)(w+1)\le (2w+2)d$.  Thus, it only remains to show that $(T, \zX')$ has adhesion at most $(2w+2)d$.

Fix an edge $t_1t_2$ of $T$, and let $T_i$ be the subtree of $T - t_1t_2$ containing $t_i$ for $i = 1, 2$.  Let $Z_i:= \bigcup_{t \in V(T_i)} X_t'$.  We want to bound the number of edges of $G$ with one end in $Z_1$ and one end in $Z_2$.  Let $z_1z_2$ be such an edge with $z_i \in Z_i$ for $i = 1, 2$.  The edge $z_1z_2$ must be contained in some bag of the tree decomposition, either in $T_1$ or in $T_2$.  It follows by the properties of a tree decomposition that either $z_1 \in X_{t_2}$ or $z_2 \in X_{t_1}$.  Thus, if we consider the bipartite subgraph of $G$ of edges with one endpoint in $Z_1$ and one endpoint in $Z_2$, we see that there does not exist a matching of size $2w+3$.  Thus, there exists a set of $2w+2$ vertices hitting all such edges and by the bound on the degree of $G$, we see that there are at most $(2w+2)d$ edges with one end in $Z_1$ and one end in $Z_2$.  Thus the adhesion of $(T, \zX')$ is at most $(2w+2)d$, completing the proof.  
\end{proof}

\section{Walls and a lower bound on the tree-cut width}

A classic theorem of the theory of minors relates the tree-width of a graph $G$ to the largest value $k$ such that $G$ contains the $k \times k$-grid as a minor.  The $k \times k$-grid has tree-width $k$.  Thus, any graph which contains the $k \times k$-grid as a minor has tree-width at least $k$.  Robertson and Seymour \cite{RS5} show that the converse is approximately true: there exists a function $w$ such that the tree-width of $G$ at least $w(k)$ contains the $k \times k$-grid as a minor.  In the next two sections, we will see that a similar result holds for the tree-cut width.  We establish the lower bound in this section and prove the upper bound in the following section.   

A \emph{wall} is a graph similar to a grid with maximum degree three.  For positive integers~$r$, define the \emph{$r$-wall} $H_r$ as follows.  Let $P_1, \dots, P_r$ 
be $r$ vertex-disjoint paths of length $r-1$.  Say for $1 \le i \le r$ that $V(P_i) = \{v_1^i\dots
v_r^i\}$ with $v_j^i$ adjacent fo $v_{j+1}^i$ for $1 \le j \le r-1$. Let $V(H_r) = \bigcup_{i=1}^r V(P_i)$, and let
\begin{equation*}
\begin{split}
E(H_r) = \bigcup_{i=1}^r E(P_i) \cup
\Big\{&v_j^i v_j^{i+1} \mid \text{ $i,j$ odd};\ 1 \le i < r;\ 1 \le j \le r\Big\} \\
& \cup \Big\{v_j^i v_j^{i+1} \mid \text{ $i,j$ even};\ 1 \le i < r;\ 1 \le j \le r\Big\}.
\end{split}
\end{equation*}
We call the paths $P_i$ the \emph{horizontal paths} of $H_r$; the paths induced by the vertices $\{v^i_j,v^i_{j+1}:1\leq i\leq r\}$ for an odd index $j$ are its \emph{vertical paths}.  Note that the graph $H_r$ has $r^2$ vertices.    See Figure~\ref{fig:h6}.
\begin{figure}[ht]
\centering
\noindent
\includegraphics[scale = .5]{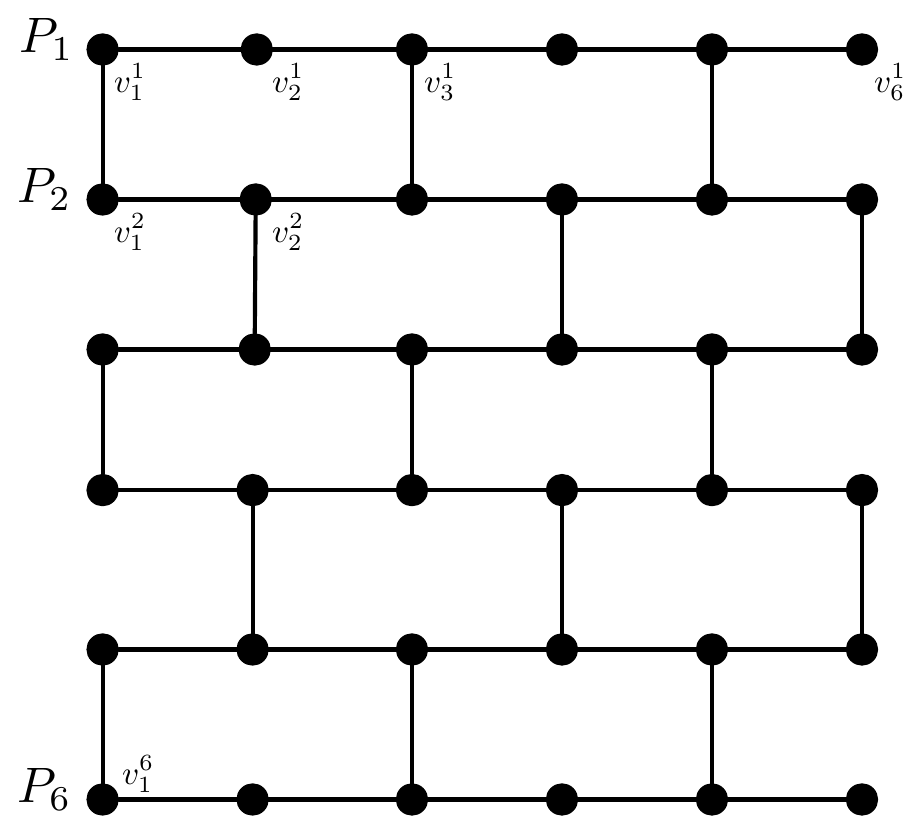}
\caption{The $6$-wall $H_6$}\label{fig:h6}
\end{figure}

It is an easy exercise to show that every graph which contains the $r \times r$ grid as a minor contains $H_r$ as a subdivision.  

We will see that every graph which admits a large wall as an immersion must have big tree-cut width.  In preparation, we prove two easy lemmas.
\begin{LE}\label{lem:tree1}
Let $T$ be a tree with $|V(T)| \ge 2$ and $X \subseteq V(T)$.  Let $k, r\ge 1$ be integers.  If $|X| \ge kr$, then one of the following must hold:
\begin{enumerate}
\item there exists a vertex $v$ such that at least $k$ components of $T-v$ contain a vertex of $X$, or
\item there exists an edge $f$ such that each component of $T-f$ contains at least $r$ vertices of $X$.
\end{enumerate}
\end{LE}
\begin{proof}
Observe that if $k=1$, the statement is trivially true.  Fix a vertex $x \in X$ and arbitrarily pick a vertex $v \in V(T)$, $v \neq x$.  Such a vertex $v$ satisfies 1.  We may assume therefore that $k \ge 2$.  For each edge $e \in E(T)$, at least one component of $T-e$ must contain at least $r$ vertices of $X$.  If both do, the statement is proven.  Thus, we may assume that there is a unique component $C_e$ of $T-e$ which has at least $r$ vertices of $X$ for every $e \in E(T)$.  We orient the edge $e$ towards the vertex which is contained in $V(C_e)$ for every edge $e \in E(T)$.    It follows that there must exist a vertex $v \in V(T)$ which has out-degree 0.  By the size of $X$, $T-v$ must have at least $k$ distinct components each of which contains a vertex of $X$, completing the proof.
\end{proof}

The next lemma is a re-statement of a lemma of \cite{FT} in terms of edge cuts and walls.  We include the proof for completeness.  
\begin{LE}\label{lem:wall1}
Let $k \ge 1, r \ge 2k+1$, be positive integers.  Let $A_1, A_2 \subseteq H_r$ be two subsets of vertices such that for every $v \in A_1 \cup A_2$, $deg(v) = 3$.  Assume $|A_1| = |A_2| = 2k^2$.  Then there does not exist a cut $U \subseteq V(H_r)$ with $A_1 \subseteq U$, $A_2 \subseteq V(H_r) \setminus U$, and $|\delta(U)| < k$.  
\end{LE}

\begin{proof}
We may assume $A_1 \cap A_2 = \emptyset$.  Since each horizontal path of the wall intersects each vertical path in at most two vertices of degree three, we see that there are either $k$ distinct horizontal paths or $k$ distinct vertical paths each of which contains a vertex of $A_i$ for $i = 1, 2$.   It is now easy to see that there cannot exist a set $X \subseteq E(H_r)$ with $|X| < k$ intersecting every $A_1-A_2$ path.  Assume such an $X$ exists.  For $i = 1, 2$, there is a path $P_i$ which is either a horizontal or vertical path of $H_r$ which is simultaneously disjoint from $X$ and contains an element of $A_i$.    As $r \ge 2k+1$, there exist at least $k$ distinct horizontal paths in $H_r$ and at least $k$ distinct vertical paths.  Thus, there exists a horizontal path $Q_h$ and a vertical path $Q_v$ of $H_r$ which are disjoint from $X$.  As $Q_v \cup Q_h$ intersects every vertical path and every horizontal path, we see that $Q_v \cup Q_h \cup P_1 \cup P_2$ is a connected subgraph, and that there exists an $A_1-A_2$ path avoiding $X$, a contradiction.  
\end{proof}

We now give the main result of this section and show that if a graph has bounded tree-cut width, then it does not admit an immersion of the $r$-wall for arbitrarily large $r$.  
\begin{theorem}\label{thm:excimersiongrid}
Let $G$ be a graph and $r\ge 3$ a positive integer.  If $G$ contains an immersion of $H_{2r^{2}}$, then $G$ has tree-cut width at least $r$.
\end{theorem}

\begin{proof}
Let $\bar{H}$ be the graph obtained from $H_{2r^{2}}$ by repeatedly suppressing all vertices of degree less than or equal to two.  Note by our choice of $r$ that $\bar{H}$ is 3-edge connected and 3-regular and $V(\bar{H}) \subseteq V(H_{2r^{2}})$.  It is easy to see that $|V(\bar{H})| \ge 2r^4$.

By Lemma \ref{lem:bdedimwidth}, it suffices to show that the graph $\bar{H}$ has tree-cut width at least $r$.  Assume, to reach a contradiction, that $(T, \zX)$ is a tree-cut decomposition of $\bar{H}$ of width at most $r-1$.  Let $Z \subseteq V(T)$ be the set of vertices $t \in V(T)$ whose corresponding $X_t$ is non-empty.  Note, as $|X_t| \le r-1$ for all $t \in V(T)$, $|Z| \ge 2r^3$.  By Lemma \ref{lem:tree1} there either exists a vertex $v \in V(T)$ such that at least $r$ components of $T-v$ contain at least one vertex of $Z$ or alternatively, there exists an edge $e$ such that each component of $T-e$ contains at least $2r^2$ vertices of $Z$.  In the first case, look at the torso $J$ of $(T, \zX)$ at the vertex $v$.  By the choice of $v$, $J$ has at least $r$ peripheral vertices.  As $\bar{H}$ is 3-edge connected, we see that the 3-center of $(J, X_v)$ is $J$, and consequently has at least $r+1$ vertices, a contradiction to the bound on the tree-cut width.  Thus, we may assume that there exists an edge $e$ such that each component of $T-e$ contains at least $2r^2$ vertices of $Z$.  This implies there exists a subset $U \subseteq V(\bar{H})$ such that both $U$ and $V(\bar{H}) - U$ each contain at least $2r^2$ vertices.  As the adhesion of the decomposition is at most $r-1$, it follows that $\delta(U) \le r-1$.  Thus, in $H_{2r^2}$ there similarly exists an edge cut of order at most $r-1$ separating two subsets $A_1, A_2$ of the vertices of degree 3 in $H_{2r^2}$, such that $|A_1| = |A_2| = 2r^2$, contrary to Lemma \ref{lem:wall1}.  This completes the proof of the theorem.
\end{proof}
The bound we obtain in Theorem \ref{thm:excimersiongrid} is almost certainly not best possible.  The correct value should be linear in $r$, as in the analogous statement about the tree-width of a grid.

\section{A grid theorem for weak immersions}

A classic theorem of the theory of minors says that a graph must either have bounded tree-width or contain a large grid minor.  Equivalently, if a graph does not contain a $k \times k$-grid minor, then the tree-width of the graph is bounded by a function of $k$.  In this section, we will prove a similar result for graphs which have bounded tree-cut width.  Our proof will use the grid minor theorem, although we will need a version based on excluded subdivisions instead of excluded minors.

\begin{theorem}[Grid minor theorem, \cite{RS5}]\label{thm:grid}
There exists a function $w= w(r)$ satisfying the following.  Let $G$ be a graph and let $r\ge1$ be a positive integer.  If the tree-width of $G$ is at least $w(r)$, then $G$ contains the $r$-wall as a subdivision.  
\end{theorem}

We now prove the analog of Theorem \ref{thm:grid} for the immersion of a large wall. The proof depends on Theorem \ref{thm:grid}.  

\begin{theorem}\label{thm:imgrid}
Let $G$ be a graph.  Let $r \ge 1$ be a positive integer.  Let $w= w(r) $ be the value given by Theorem \ref{thm:grid}.   If $G$ has tree-cut width at least $4r^{10} w(r)$, then $G$ admits a weak immersion the $r$-wall.
\end{theorem}

\begin{proof}
The theorem trivially holds for $r = 1, 2$, so we may assume that $r \ge 3$.
Assume the theorem is false and let $G$ be a graph which has tree-cut width at least $4r^{10} w(r)$ and does not admit a weak immersion of the $r$-wall.  Moreover, assume $G$ is chosen from all such counterexamples to minimize $|V(G)|$.  

\begin{claim}
$G$ has $(r^2,r^4)$-bounded degree.   
\end{claim}
\begin{cproof}
Assume otherwise.  The graph $G$ cannot admit a weak immersion of $K_{r^2}$, lest $G$ admit an immersion of the $r$-wall as well.  Thus, by Theorem \ref{thm:structweak}, $G$ is a grounded $k$-edge sum of two graph $G_1$ and $G_2$ for some $k \le r^2$.  If either $G_1$ or $G_2$ admitted the $r$-wall as a weak immersion, then by Lemma \ref{lem:sum}, $G$ would as well.  Thus, by the minimality of $G$, both $G_1$ and $G_2$ have tree-cut width strictly less than $4r^{10} w(r)$.  But now by Lemma \ref{lem:bdwidthsum}, $G$ has tree-cut width strictly less than $4r^{10} w(r)$, a contradiction.  
\end{cproof}

Let $Z$ be the set of vertices of degree at least $r^4$.  Let $n \ge 1$ be a positive integer, and let $H_1, \dots, H_n$ be the connected components of $G-Z$.  Then for all $i$, $H_i$ has maximum degree at most $r^4-1$.  By Theorem \ref{thm:grid}, the subgraph $H_i$ has tree-width at most $w(r)$ for all $1 \le i \le n$.

We now see that there are a bounded number of edges with one end in $H_i$ and the other end in $Z$ for all indices $i$.

\begin{claim}\label{cl:2}
For all $i \le n$, there exist at most $ 3r^{10}w(r)$ edges with one end in $Z$ and one end in $V(H_i)$.
\end{claim}

\begin{cproof}
Fix an index $i \le n$.  Given a vertex $z \in Z$ and a subset $Y \subseteq V(H_i)$, we say $z$ is \emph{triconnected} to $Y$ if there exist three distinct edges with one endpoint equal to $z$ and other endpoint contained in $Y$. Let 
\begin{equation*}
\zY = \{Y \subseteq V(H_i): H_i[Y]\text{ is connected and }\exists~z \in Z \text{ such that $z$ is triconnected to $Y$}\}.
\end{equation*}
Note that for each $Y \in \zY$, there exist vertices $y \in Y$ and $z \in Z$ and paths $P_1, P_2, P_3$ contained in $G[Y \cup Z]$ such that $P_1, P_2, P_3$ have $y$ and $z$ as common endpoints, no internal vertex in $Z$, and are pairwise edge-disjoint.  

Fix a tree decomposition $(T, \zX)$ of $H_i$ of width at most $w(r)$.  For every $Y \in \zY$, let $T(Y)$ be the subgraph of $T$ induced by the vertex set $\{t \in V(T): X_t \cap Y \neq \emptyset\}$.  By the fact that $H_i[Y]$ is connected, it follows that $T(Y)$ is a subtree of $T$ for all $Y \in \zY$.  It is a standard exercise to show that either there exist $m=r^4$ distinct elements $Y_1, \dots, Y_m$ of $\zY$ such that the trees $T(Y_j)$ and $T(Y_{j'})$ are vertex-disjoint for $j \neq j'$, or alternatively, there exists a set of at most $r^4$ vertices of $T$ intersecting $T(Y)$ for all $Y \in \zY$.  

Assume, as a case, that there exist such distinct $Y_1, \dots, Y_m$ whose corresponding subtrees of $T(Y_j)$ of $T$ are pairwise disjoint.  For each $j$, $1 \le j \le m$, there exists a vertex $y_j \in Y_j$ and three edge-disjoint paths $P_1^j, P_2^j, P_3^j$ contained in $G[Y_j \cup Z]$ such that $P_1^j$, $P_2^j$, and $P_3^j$ have a common endpoint in $Z$, another common endpoint equal to $y_j$, and no internal vertex in $Z$.  Thus, given the bound on $|Z|$, there exists a vertex $z \in Z$ such that $r^2$ of the indices $j$, $1 \le j \le m$, have their corresponding paths $P_1^j$ terminating at the same vertex $z\in Z$.  By construction, the paths $P_{j'}^j$ and $P_{l'}^l$ are edge-disjoint for $1 \le j', l' \le 3$, $1 \le j < l \le m$.  We see that $G$ thus admits an immersion of $S_{3, r^2}$, and consequently by Observation \ref{ob:multistar}, admits the $r$-wall as an immersion as well.

We conclude that there exist at most $r^4$ vertices of $T$ intersecting $T(Y)$ for all $Y \in \zY$.  Thus, there exists a subset $Z' \subseteq V(H_i)$ of size at most $r^4(w(r)+1)$ intersecting $Y$ for all $Y \in \zY$.  Note that by the bound on the maximum degree of $H_i$, there are at most $(r^4-1)(r^4(w(r)+1)) \le r^8w(r)$ components of $H_i - Z'$. Every vertex $z \in Z$ has at most two incident edges with an end in each component of $H_i - Z'$.  Thus, there are at most $2r^{8}w(r)(r^2) = 2r^{10}w(r)$ edges with one endpoint in $V(H_i) - Z'$ and the other end in $Z$.  As every vertex in $H_i$ has degree (in $G$) at most $r^4-1$, there are at most $(r^4-1)(r^4(w(r)+1)) \le r^8w(r)$ edges with one end in $Z'$ and the other end in $Z$.  Thus, there are a total of at most $r^8w(r) + 2r^{10}w(r) \le 3r^{10}w(r)$ edges with one end in $V(H_i)$ and the other end in $Z$, as desired.
\end{cproof}

\begin{claim}\label{cl:4}
There are at most $3r^{8}$ distinct indices $i$ such that there are at least 3 edges with one end in $H_i$ and one end in $Z$.
\end{claim}
\begin{cproof}
Assume, to reach a contradiction, that there exists a set $I$ of size at least $3r^{8}$ such for all $i \in I$, there are at least 3 edges with one end in $H_i$ and one end in $Z$. For each $i \in I$, there exists a vertex $v_i \in V(H_i)$ and 3 edge-disjoint paths in $G[V(H_i) \cup Z]$, each with one end equal to $v_i$, one end in $Z$, and no internal vertex in $Z$.  The endpoints of these paths form a multi-set of the elements of $Z$ of size 3.  As there are at most $(r^2)^3$ such possible multi-sets, it follows that there exists a multi-set $X \subseteq Z$, $|X| = 3$, and a subset  $I' \subseteq I$ with $|I'| \ge 3r^2$ such that for each $i \in I'$, the vertex $v_i$ has 3 pairwise edge-disjoint paths with endpoints in $Z$ equal to the three elements of $X$.  If we fix an element $x \in X$, we see that for each set of 3 distinct indices $I''$ in $I'$, we can fix a vertex $v_i$ for some $i \in I''$ and find 3 pairwise edge-disjoint paths from $v_i$ to the vertex $x$.  We conclude that $G$ contains an immersion of $S_{3, r^2}$, and consequently the $r$-wall as an immersion, a contradiction.  
\end{cproof}

We now construct the desired decomposition of $G$.  For each $i$, $1 \le i \le n$, $H_i$ has a tree-cut decomposition $(T_i, \zX_i)$ of adhesion at most $(2w(r)+2)r^4$ with $\zX_i = \{X^i_t: t \in V(T_i)\}$ such that every torso has size at most $(2w(r)+2)r^4$ by Lemma \ref{lem:bdedtw}.  Arbitrarily fix a vertex $v_i \in V(T_i)$ for each $1 \le i \le n$, and now define the tree $T$ as the union of the trees $T_i$ along with a vertex, call it $v$, adjacent to $v_i$ for all $1 \le i \le n$.  We define $\zX = \{X_t: t \in V(T)\}$ as follows.  Let $t \in V(T)$.  If $t \in V(T_i)$ for some index $i$, we let $X_t = X^i_t$.  Otherwise, $t = v$ and we let $X_t = Z$.  We claim $(T, \zX)$ is the desired decomposition of $G$.

First, we see that the adhesion of $(T, \zX)$ is bounded. For every edge $e \in E(T)$, either $e$ has one end equal to the vertex $v$, or $e \in E(T_i)$ for some index $i$.  By Claim \ref{cl:2}, if $e$ has one end equal to $v$, then there are at most $3r^{10}w(r)$ edges traversing the cut corresponding to $e$.  Alternatively, if $e$ is contained in $T_i$ for some index $i$, then the edges traversing the cut corresponding to $e$ are either contained in $H_i$, or alternatively, have one end in $H_i$ and one end in $Z$.  It follows that there are at most $(2w(r)+2)r^4 + 3r^{10}w(r) \le 4r^{10}w(r)$ edges traversing the cut corresponding to $e$.  Thus, $(T, \zX)$ has adhesion at most $4r^{10}w(r)$.

Consider a vertex $t \in V(T)$ such that $t \in V(T_i)$ for some index $i$.  By Lemma \ref{lem:bdedtw}, the torso of $(T_i, \zX_i)$ at $t$ has at most $(2w(r)+2)r^4$ vertices.  As $t$ has at most one more neighbor in $T$ than in $T_i$, it follows that the torso $J$ of $(T, \zX)$ at $t$ has at most $(2w(r)+2)r^4 + 1$ vertices.  Thus, we can bound size the 3-center of $(J, X_t)$ by $3r^4w(r)$ given that $w(r) \ge w(3) \ge 4$.  If we consider the vertex $v \in V(T)$, we see that the vertex set of the torso $J_v$ consists of the set $Z$ of vertices along with one peripheral vertex for each contracted $H_i$, call it $h_i$.  Note, by construction, no edge has endpoints in distinct $H_i$.  Thus, the edges of the torso at $v$ are either edges of $G[Z]$ or edges with one end in $Z$ and one end equal to $h_i$ for some index $i$.  By Claim \ref{cl:4}, there are at most $3r^{8}$ distinct indices $i$ such that there are at least 3 edges with one end in $H_i$ and one end in $Z$.  Thus, there are at most $3r^{8}$ distinct indices $i$ such the vertex $h_i$ has degree at least 3 in $J_v$.  It follows that the 3-center of $(J_v, Z)$ has at most $3r^{8}+ r^2 \le 4r^{10}w(r)$ vertices.  We conclude that the width of the tree-cut decomposition $(T, \zX)$ is at most $4r^{10}w(r)$, completing the proof of the theorem.
\end{proof}

It is unclear whether Theorem \ref{thm:imgrid} can be proven without the dependence on Theorem \ref{thm:grid} without essentially replicating the proof of the grid minor theorem.  

Reed and Wood \cite{RW} have shown that there exists a polynomial $f = f(r)$ such that every graph of tree-width at least $f(r)$ contains what they call a \emph{grid-like graph} of order $r$ as a subgraph.  While we will not need the exact definition here, the edges of a \emph{grid-like graph} can be partitioned into two sets $\zP$ and $\zQ$ of paths such that the elements of $\zP$ are pairwise vertex-disjoint and the elements of $\zQ$ are also pairwise vertex-disjoint.  Moreover, the grid-like graph of order $r$ has tree-width at least $\lceil r/2 \rceil -1$.  Thus, an immediate consequence of this result is that polynomial tree-width (in $r$) suffices to force the existence of a subgraph with maximum degree four of tree-width $r$.  Moreover, the proof is quite short and elegant.  

Given that an immersion of a large $r$-wall in a graph of max degree three must contain a subdivision of a large wall, we might consider whether there exists a short and relatively easy proof that every graph of sufficiently large tree-width (as a function of $r$) contains a subgraph with maximum degree three and tree width $r$.  By Lemma \ref{lem:bdedtw}, we know that graphs with bounded degree and large tree-width also have large tree-cut width.  Thus, if we could easily show that every graph of sufficiently large tree-width contains a subcubic graph of large tree-width, then any proof that sufficiently large tree-cut width implies the existence of a wall immersion would give an alternate proof of the grid minor theorem.

\pagebreak

\begin{center}
{\bf ACKNOWLEDGEMENTS}
\end{center}

The author gratefully acknowledges Paul Seymour's valuable contributions at the early stages of this project.  Theorem \ref{thm:imgrid} benefitted significantly from his input, and he had an important part in finding the right definition for the width of a tree-cut decomposition.   We are also grateful for his helpful feedback on several early drafts.  We thank as well the anonymous referees whose input greatly improved the article.


\begin{thebibliography}{99}
{\footnotesize

\bibitem{AL} F.~Abu-Khzam, M. Langston, Graph Coloring and the Immersion Order, {\it Lecture Notes in Computer Science} {\bf 2697} (2003) 394 -- 403.

\bibitem{CS} F.R.K. Chung, P.D. Seymour, Graphs with small bandwidth and cutwidth, {\it Discrete Math.} {\bf 75} (1:3) (1989)
113--119.


\bibitem{DeVos} M.~DeVos, J.~McDonald, B.~Mohar, D.~Scheide, A Note on Forbidding Clique Immersions. {\it Electr. J. Comb.} {\bf 20}(3): P55 (2013).


\bibitem{DDFMMS} M. DeVos, Z. Dvorak, J. Fox, J. McDonald, B. Mohar, D. Schiede, Minimum degree condition forcing complete graph immersion, {\it Combinatorica} {\bf 34}(3) (2014), 279--298.

\bibitem{DKMO} M. DeVos, K. Kawarabayashi, B. Mohar, and H. Okamura, Immersing small complete
graphs, {\it Ars Math. Contemp.} {\bf 3} (2010), 139--146.

\bibitem{D} Z. Dvorak, {\it personal communication}.

\bibitem{FL} M.R. Fellows, M.A. Langston, 
On well-partial-order theory and its application to combinatorial problems
of VLSI design, {\it SIAM J. Discrete Math.} {\bf 5} (1) (1992) 117--126.

\bibitem{FGTW} M. Ferrara, R. Gould, G. Tansey, T. Whalen, 
On $H$-immersions, 
{\it J. Graph Theory} {\bf 57} (2008) 245--254.

\bibitem{FT} J.~O.~Fr\"ohlich, T. M\"uller, 
Linear connectivity forces large complete bipartite minors: An alternative approach , 
{\it J. Combin. Theory, Ser. B} {\bf 101} (6) (2011) 502--508.

\bibitem{GJ} M.R. Garey, D.S. Johnson,
 Computers and Intractability. A Guide to the Theory of NP-Completeness, 
 Freeman, San Francisco, CA, 1979.
 
\bibitem{HLMP} P. Heggernes, D. Lokshtanov, R. Mihai, C. Papadopoulos, 
Cutwidth of Split Graphs and Threshold Graphs, {\it SIAM J. Discrete Math.} {\bf 25} (3) (2011) 1418--1437

\bibitem{KT} A.C. Giannopoulou, M. Kaminski, D. M. Thilikos, Forbidding Kuratowski graphs as immersions, {\it Journal of Graph Theory} in press (2014).

\bibitem{KK} K. Kawarabayashi, Y. Kobayashi, List-coloring graphs without subdivisions and without immersions,
 Proceedings of the Twenty-Third Annual ACM-SIAM Symposium on Discrete Algorithms
SIAM (SODA 12) (2012).
 
\bibitem{LM} F. Lescure, H. Meyniel, On a problem upon configurations contained in graphs with given chromatic number, Graph theory in memory of G. A. Dirac (Sandbjerg, 1985),  {\it Ann. Discrete Math.} {\bf 41} (1989) 325--331,.

\bibitem{N} C. St. J. A. Nash-Williams, On well-quasi-ordering trees, Theory of Graphs and Its Applications (Proc. Symp. Smolenice, 1963), Publ. House Czechoslovak Acad. Sci., 1964, 83--84.

\bibitem{RW} B. Reed, D. Wood, Polynomial tree-width forces a large grid-like-minor, {\it J. European Comb.} {33} (3) (2012) 374--379.

\bibitem{RS5} N. Robertson, P.D. Seymour, Graph minors V: Excluding a planar graph, {\it J. Combin. Theory, Ser. B} {\bf 41} (1) (1986) 92--114.

\bibitem{RS23} N. Robertson, P.D. Seymour, Graph minors XXIII: Nash-Williams immersion conjecture, {\it J. Combin. Theory, Ser. B} {\bf 100} (2) (2010) 181--205.


\bibitem{ST} P.D. Seymour, R. Thomas, Call routing and the ratcatcher, {\it Combinatorica} {\bf 14} (2) (1994) 217--241

\bibitem{TSB1} D.M. Thilikos, M. Serna, H.L. Bodlaender, Cutwidth I: A linear time fixed parameter algorithm,
{\it J. Algorithms} {\bf 56} (1) (2005) 1--24.

\bibitem{TSB2} D.M. Thilikos, M. Serna, H.L. Bodlaender, Cutwidth II: Algorithms for partial w-trees of bounded degree,
{\it J. Algorithms} {\bf 56} (1) (2005) 25--49.

\bibitem{Y} M. Yannakakis, A polynomial algorithm for the min-cut linear arrangement of trees, {\it J. ACM} {\bf 32} (4) (1985)
950--988.

}
\end{thebibliography}
\end{document}